\documentclass[a4paper,10pt]{article}
\usepackage{xcolor}
\usepackage{enumitem}
\usepackage{graphicx}
\usepackage{tkz-euclide}
\usepackage[margin=0.9in]{geometry} 

\usepackage{cite}
\usepackage{authblk}
\usepackage{amssymb}
\usepackage{amsthm}
\usepackage{amsmath}
\usepackage{color}
\usepackage[utf8]{inputenc}

\usepackage{tikz}
\usepackage{pgfplots}
\usetikzlibrary{decorations.markings, intersections, fillbetween}

\usetikzlibrary{calc}
\def\centerarc[#1](#2)(#3:#4:#5)
    { \draw[#1] ($(#2)+({#5*cos(#3)},{#5*sin(#3)})$) arc (#3:#4:#5); }
\usetikzlibrary{arrows.meta}
\tikzset{
  pics/carc/.style args={#1:#2:#3}{
    code={
      \draw[pic actions] (#1:#3) arc(#1:#2:#3);
    }
  }
}
\tikzset{mylabel/.style  args={at #1 #2  with #3}{
    postaction={decorate,
    decoration={
      markings,
      mark= at position #1
      with  \node [#2] {#3};
 } } } }
\usepackage[all]{xy}
\usepackage[bookmarksnumbered,colorlinks]{hyperref}
\usepackage[normalem]{ulem}
\usepackage{float} 
\def\br#1\er{\textcolor{red}{#1}} 
\author[1]{José Luis Flores}
\author[2]{J\'onatan Herrera}
\author[3]{Didier A. Solis}

\affil[1]{Universidad de Málaga}
\affil[2]{Universidad de Córdoba}
\affil[3]{Universidad Autónoma de Yucatán}

\title{Low regularity approach to Bartnik's conjecture}

\newtheorem{thm}{Theorem}[section]
\newtheorem{theorem}{Theorem}[section]
\newtheorem{proposition}[thm]{Proposition} 
\newtheorem{corollary}[thm]{Corollary} 
\theoremstyle{definition} \newtheorem{definition}[thm]{Definition}

\newtheorem{remark}[thm]{Remark}

\begin{document}
\maketitle
\begin{abstract}
  In this work we establish a version of the Bartnik Splitting Conjecture in the context of Lorentzian length spaces. {In precise terms, we show that under an appropriate timelike completeness condition, a globally hyperbolic Lorentzian length space of the form $\Sigma\times \mathbb{R}$ with $\Sigma$ compact splits as a metric Lorentzian product, provided it has non negative timelike curvature bounds.} This is achieved by showing that the causal boundary of that Lorentzian length space consists on a single point. 
\end{abstract}

\section{Introduction}

Over the last decade or so there has been a growing interest in developing frameworks tailored to handle problems in general relativity in the absence of a regular --that is, at least $C^2$-- Lorentzian structure. The wealth of tools employed to this effect include the study of rough metrics \cite{CG,Graf,KOSS}, the use of distributional methods \cite{PS}, cone structures \cite{Mingcone} and ordered sets \cite{Surya,Olafdim}.  The broad extent and diversity of these methods may prove valuable in bringing a deeper understanding to singular phenomena such as the provided by recent observations of black holes and gravitational waves \cite{LIGO,M87}.   

Among these alternatives, the so-called synthetic methods have proven very successful in extending classical results to non-smooth contexts. Motivated by the well established theory of metric length spaces \cite{Burago}, several different approaches are currently being pursued. For instance, the theory of Lorentzian (pre)-length spaces, first studied by Kunzinger and S\"amman 
\cite{KunzingerSamann2018} provides a setting in which several different notions of comparison angles  and timelike curvature bounds can be defined \cite{BarreraEtal2022,BeranSamann2023,timelikecurv}, in ways analogous to the ones developed in the contexts of (metric) Alexandrov\textcolor{red}{\footnote{Recall that an \emph{Alexandrov space}  is a locally compact, complete, and path-connected length space $(X, d)$ with curvature bounded below by $k$.}}  and CAT spaces \cite{AKP,Bridson}. In this setting, landmark results as the causal hierarchy \cite{ACS,causalgluing}, Bonnet theorem \cite{Bonnet} and Toponogov splitting theorem \cite{BeranEtal2023} have been obtained. Moreover, the related notion of almost Lorentzian pre-length space has been useful in studying Lorentzian structures from a functorial point of view \cite{Olaf}. On the other hand, the notion of Lorentzian metric space was introduced in \cite{MS} in order to attain convergence results in the low regularity setting. slight variation of this concept, closer to the one of pre-length space, has led to a description of the causal boundary in the rough setting \cite{BFH} . It is worthwhile noting that synthetic methods in Lorentzian geometry have also paved the way to explore diverse geometrical objects, such as contact structures \cite{Hedicke} and hyperspaces \cite{BMS2}. 

One of the famous standing problems in mathematical relativity is the so called \textit{Bartnik Splitting Conjecture} (BSC). Originally posed in \cite{Bartnik}, it can be interpreted as the rigidity associated to the generic condition in the Hawking-Penrose singularity theorem. In precise terms it reads as follows:

\medskip

\noindent\textbf{Conjecture} Let $(M,g)$ be a spacetime that satisfies
\begin{enumerate}
\item $(M,g)$ is globally hyperbolic with compact Cauchy surfaces.
\item $(M,g)$ satisfies the Strong Energy Condition: $Ric (X,X)\ge 0$ for all timelike $X$.
\end{enumerate}
If $(M,g)$ is timelike geodesically complete then $(M,g)$ splits isometrically as $(N\times\mathbb{R}, h-dt^{2})$, where $(N,h)$ is a compact Riemannian manifold.

\medskip

As pointed out by Bartnik himself, the BSC is true provided there exists a (necessarily compact) Cauchy surface with constant mean curvature.\footnote{In fact, due to Hawking's Singularity Theorem, such Cauchy hypersurface must have vanishing mean curvature.}  Subsequently, different additional hypotheses have been considered in order to guarantee the existence of such a hypersurface. In this sense,
recall that an \textit{observer (particle) horizon} is a set of the form $\partial I^-(\gamma )$ ($\partial^+(\gamma )$) where $\gamma$ is a future (past) inextendible timelike curve. 
The non existence of such horizons, i.e. the condition $I^+(\gamma )=I^-(\gamma)=M$, is equivalent to the compactness of the sets $\partial I^\pm (p)$, $p\in M$ \cite{Gal84}. Under de hypotheses of the BSC, this condition, which is commonley referred to in the literature as the {\em no horizon condition}, guarantees the existence of some CMC Cauhy hypersurface passing through $p$. So, under the non-positive sectional curvature condition $K\le 0$, which implies both, 
the Strong Energy Condition and the no horizon condition, the BSC follows.

An alternative proof for the BSC circumventing the use of CMC hypersurfaces can be pursued using the Lorentzian Splitting Theorem \cite{Esch88}. Indeed, the compactness of the Cauchy surface enables us through standard methods to assert the existence of a \textit{causal} line, which in the presence of the sectional curvature bounds can be shown to be timelike, and the mentioned splitting result does the rest \cite{CSHF22}. 

The aim of this work is to find an appropriate synthetic version of BSC. Note however that the notion of hypersurface has not been well established in the synthetic scenario, nor an extrinsic theory. Consequently, no adequate definition of CMC exists for Lorentzian length spaces. This joined to the recent proof of the synthetic version of the Lorentzian Splitting Theorem \cite{BeranEtal2023} makes the second approach described above a more suited framework than the first one to study the BSC. In this setting, a few remarks are in place. 

First, we note that in sharp contrast to the smooth setting \cite{Geroch}, a globally hyperbolic (pre-)length space need not split topologically as a product \cite{BH}, thus we need to restrict our attention to product spaces. Further, for Lorentzian (pre-)length spaces there is no clear choice of a family of preferred parameters --in the spirit of affine parameters-- for the class of timelike maximal curves in general, thus there is no natural notion of completeness in this context. Alternatively, we can establish a general formulation in terms of asympotic properties of the Lorentzian distance function.

In regards to the curvature condition, we rely on the synthetic notion first introduced in \cite{KunzingerSamann2018} and further developed in \cite{BarreraEtal2022,timelikecurv,BeranSamann2023}. This in turn is based upon the work of  Alexander and Bishop \cite{AB} on curvature bounds for semi-Riemannian manifolds. Notice that due to the use of signed distances in this seminal study, \textit{non-positive} sectional curvature for timelike planes corresponds to \textit{non-negative} timelike curvature bounds.

Finally, the no horizon condition is closely related to the \textit{causal boundary}, --or \textit{c-boundary} for short-- first introduced in \cite{GKP}. Intuitively, the c-boundary is made of the ideal points associated to families of future or past inextensible timelike curves. The connection between these concepts is that the no horizon condition holds if and only if the future (past) causal boundary consists of a singleton. Since the c-completion of a Lorentzian pre-length space has been recently explored \cite{ABS,BFH}, the heart of our approach consists of showing that, under the sectional curvature bound, the c-boundary is made of only one point.

Our main result is the following:

\begin{theorem}
  \label{thm:main:2}
  Let $(X,d,\ll,\leq,\tau)$ be a connected regularly localisable, globally hyperbolic Lorentzian length space with proper metric $d$ and global non-negative timelike curvature satisfying timelike geodesic prolongation. Assume that $X=\Sigma\times \mathbb{R}$ with $\Sigma$ compact, {and that the vertical curves $t\mapsto (x,t)$  are timelike complete with $\tau((x,t),(x,s))>0$ for any $t,s\in \mathbb{R}$.} Then, there exists a $(\tau,\leq)$-preserving homeomorphism $f:S\times \mathbb{R}\rightarrow X$, where $S$ is a proper, strictly intrinsic metric of Alexandrov curvature $\geq 0$.
\end{theorem}

The paper is organized as follows. In section \ref{sec:preliminaries} we state the main tools and establish the notation that will be used throughout this work. In section \ref{sec:cbdry} we analyze the structure of the causal boundary. Finally, in section \ref{sec:main} we give the proof of Theorem \ref{thm:main:2}.

\section{Preliminaries}\label{sec:preliminaries}

\subsection{Lorentzian metric spaces and its c-completion}

In this section we will introduce the definitions and results that we will require along the paper. We begin with the notion of Lorentzian metric space.

\begin{definition}
  \label{def:main:1}
  A \emph{Lorentzian metric space} is a triple $(X,\sigma,\tau)$ where $(X,\sigma)$ is a topological set, and $\tau:X\times X\rightarrow [0,\infty]$ is a function satisfying:

  \begin{enumerate}[label=(\roman*)]
  \item $\tau$ is lower-semicontinuous with respect to $\sigma$;
    
  \item the reverse triangle inequality holds, i.e.
    \begin{equation}
      \label{eq:1}
      \tau(x,z)\geq \tau(x,y)+\tau(y,z),\qquad \hbox{if $\;\tau(x,y),\;\,\tau(y,z)>0$.}
    \end{equation}
  \end{enumerate}
\end{definition}
Given a Lorentzian metric space $(X,\sigma,\tau)$, we can define a \emph{chronological relation} in the following way: two points $x,y\in X$ are chronologically related $x\ll y$ if and only if $\tau(x,y)>0$. Now, we can introduce some basic elements from causality theory. A {\em chronological chain} is a sequence of points $\{x_n\}\subset X$ such that $x_n\ll x_{n+1}$ for all $n$. For a given set $S\subset X$ we can define $I^-(S):=\left\{ y\in X: \exists\; x\in S \hbox{ with } y\ll x \right\}$. We will say that a set $S$ is a {\em past set} if $I^-(S)=S$. An {\em indecomposable past set}, or {\em IP}, is a past set that cannot be decomposed into two proper past sets. We will say that an IP $P$ is a {\em proper IP}, or {\em PIP} for short, if there exists a point $x\in X$ so $P=I^{-}(x)$. The rest of IPs will be called {\em terminal IPs,} or {\em TIPs} for short. The set of all IPs will be called the {\em future causal completion}, and it will be denoted by $\hat{X}$. {The set $\hat{\partial}X:=\hat{X}-X$ is the \textit{ future causal boundary.}}  In analogy, the notions for future sets, IFs, etc. are given by time-reversion. We will name \emph{past chronological completion}, denoted $\check{X}$, to the set of all IFs {and $\check{\partial}X:=\check{X}-X$ the \textit{past causal boundary}}.

Between the elements of $\hat{X}_{\emptyset}:=\hat{X}\cup \left\{\emptyset \right\}$ and $\check{X}_{\emptyset}:=\check{X}\cup \left\{ \emptyset  \right\}$, it is possible to define the so-called $S$-{\em relation} (due to Szabados), given in the following way: an IP and an IF are $S$-related if

\begin{eqnarray}
  \label{eq:2}
  F\subset \uparrow P:=\left\{ y\in X: x\ll y \hbox{ for all $x\in P$} \right\}\hbox{, and $F$ is a maximal IF into $\uparrow P$.}\\
  P\subset \downarrow F:=\left\{ y\in X: y\ll x \hbox{ for all $x\in P$} \right\}\hbox{, and $P$ is a maximal IP into $\downarrow F$.}
\end{eqnarray}
If for a given IP $P$ there is no $F\neq \emptyset$ $S$-related to it, we will say that $P\sim_S\emptyset$ (an analogous definition follows for $\emptyset\sim_S F$).

\medskip

Let us introduce now some mild conditions for a Lorentzian metric space (see \cite{BFH}), which will become useful later.

\begin{definition}
  \label{def:main:2}
  A Lorentzian metric space $(X,\sigma,\tau)$ is:
  \begin{enumerate}[label=(\roman*)]
  \item \emph{chronologically dense} if for every point $x\in X$ with $I^-(x)\neq \emptyset$ (resp. $I^+(x)\neq \emptyset$) there exists a future (resp. past) chronological chain $\left\{ x_n \right\}$ with topological limit $x$;
  \item \emph{distinguishing} if 
    $I^-(x)=I^-(y)$ implies $x=y$ (\emph{past-distinguising}) and $I^+(x)=I^+(y)$ implies $x=y$ (\emph{future-distinguising}).
  \item \emph{strongly causal} if it is distinguishing and the space topology $\sigma$ agrees with the \emph{Alexandrov topology}, i.e. the topology generated by the subbasis of chronological diamons $I(x,y):=I^+(x)\cap I^-(y) $ for all $x,y\in X$.
  \end{enumerate}
\end{definition}

We are now in conditions to introduce the notion of c-completion :
\begin{definition}\cite[Sec. 5.2]{BFH}
  \label{def:main:3}
  The {\em c-completion} of a  
  chronologically dense and strongly causal Lorentzian metric space
  $(X,\sigma,\tau)$ is another Lorentzian metric space $(\overline{X},\overline{\sigma},\overline{\tau})$, where:
  \begin{itemize}
  \item $\overline{X}:=\left\{ (P,F)\in \hat{X}_{\emptyset}\times \check{X}_{\emptyset}:P\sim_S F  \right\}$
  \item $\overline{\sigma}$ is the {\em chronological topology}, defined as the sequential topology associated to the limit operator $L$; i.e., a subset $C$ is {\em closed} for $\sigma_{chr}$ if and only if ${L}(\sigma) \subset C$ for any sequence $\sigma \subset C$, where $$(P,F) \in L(\sigma) \Leftrightarrow P \in \hat{L}(P_n) \text{ if } P\neq \emptyset, \text{ and } F \in \check{L}(F_n) \text{ if } F \neq \emptyset,$$
    with
    \begin{align*}
      P\in \hat{L}(P_n) &\Leftrightarrow \left\lbrace \begin{array}{c}
        P\subset LI (P_n) \\
        P \text{ is a maximal IP into } LS(P_n),
      \end{array} \right. \\
      F\in \check{L}(F_n) &\Leftrightarrow \left\lbrace \begin{array}{c}
        F\subset LI (F_n) \\
        F \text{ is a maximal IF into } LS (F_n),
      \end{array} \right.
    \end{align*}
    and being $LI$ and $LS$ the usual inferior and superior limit operators in set theory;
  \item $\overline{\tau}((P,F),(P',F')):=\lim_{n\rightarrow \infty}\tau(y_n,x'_n)$ for any past-directed chain $\{y_n\}$ generating $F$, and any future-directed chain $\{x'_n\}$ generating $P'$ (if either $F=\emptyset$ or $P'=\emptyset$, $\overline{\tau}((P,F),(P',F'))=0$). 
  \end{itemize}
\end{definition}
With this definition, the following properties hold:
\begin{enumerate}[label=(\roman*)]
\item $X$ is topologically embedded  and dense in the c-completion $\overline{X}$.
\item $\overline{\tau}$ coincides with $\tau$ when it is evaluated on points of $X$.
\item $\overline{X}$ is {\em chronologically complete} in the sense that any chronological chain in $\overline{X}$ has a topological endpoint. 
\end{enumerate}

Our main result (Theorem \ref{thm:main:2}) involves \emph{globally hyperbolic} Lorentzian metric spaces. In the literature there are several ways to define the notion of global hyperbolicity. Here, we will consider the one given by the compactness of the causal diamons, which requires a causal relation on $(X,\sigma,\tau)$. More precisely:

\begin{definition}
  \label{def:main:13}
  Let $(X,\sigma,\tau)$ be a Lorentzian metric space, and let $\leq$ be a pre-order relation defined on $X$. We say that $\leq$ is \emph{compatible with} $X$ if:

  \begin{enumerate}[label=(\roman*)]
  \item $x\leq y$ whenever $\tau(x,y)>0$, and
  \item \label{compatible2} $\tau(x,z)\geq \tau(x,y)+\tau(y,z)$ whenever $x\leq y\leq z$.
  \end{enumerate}
\end{definition}
\noindent Note that condition \ref{compatible2} ensures that the triangle inequality, valid for chronologically related triples, can be extended now to causally related ones. Moreover, this condition implies the so-called \emph{push-up property} (see \cite[Proposition 3.6]{BFH}); i.e. for any $x,y,z\in X$ satisfying $x\leq y\ll z$ or $x\ll y\leq z$, necessarily $x\ll z$. On the other hand,
with a causal relation in place we can define the  \emph{causal} past of a set $S$ as $J^{-}(S)=\left\{y\in X: \exists\; x\in S \hbox{ with }y\leq x  \right\}$, being the definition for future analogous.
Now, we can give the following definitions:

\begin{definition}
  \label{def:main:14}
  A Lorentzian metric space $(X,\sigma,\tau)$ joined to a compatible causal relation $\leq$ is said
  \begin{enumerate}[label=(\roman*)]
  \item {\em (globally) causally closed} if the causal relation $\leq$ is closed, i.e., if given two convergent sequences $x_n\rightarrow x$  and $y_n\rightarrow y$ with $x_n\leq y_n$ necessarily $x\leq y$; 
  \item {\em causally simple} if it is strongly causal and  $\overline{I^{\pm}(x)}=J^{\pm}(x)$ for any $x\in X$;
  \item {\em globally hyperbolic} if it is strongly causal and the set $J(x,y):=J^+(x)\cap J^-(y)$ is compact for any $x,y\in X$.
  \end{enumerate}
\end{definition}
\begin{remark} It is clear from previous definitions that global hyperbolicity implies causal simplicity and causal closedness {provided the sets $I^\pm (x)$ are not empty.}
\end{remark}

\subsection{Lorentzian length spaces}
In this section we will complete the structure introduced in the previous one. Most of the notions and results of this section are taken from \cite{BeranSamann2023}. Our first required definition is the concept of Lorentzian pre-length space:

\begin{definition}
  \label{def:main:4}
  A Lorentzian pre-length space is a {causal space} $(X,\ll,\leq)$ together with a metric $d$ on $X$ and a map $\tau:X\times X\rightarrow [0,\infty]$ satisfying:
  \begin{enumerate}[label=(\roman*)]
  \item $\tau$ is lower semi-continuous with respect to $d$,
  \item $\tau(p,r)\geq \tau(p,q)+\tau(q,r)$ for any $p\leq q\leq r$ and
  \item $\tau(p,q)>0$ iff $p\ll q$.
  \end{enumerate}
\end{definition}
\noindent In summary, a Lorentzian pre-length space is a Lorentzian metric space with a compatible causal relation and a metrizable topology $\sigma$ associated to $d$. 

{An example of Lorentzian (pre-)length space that will be most relevant in this work is the \textit{Lorentzian metric product}. Given the metric space $(\Sigma, d)$, define on the cartesian product $X = \Sigma\times \mathbb{R}$ the chronological and causal relations as
  \begin{eqnarray*}
    (x, s)\ll (y, t) &\Leftrightarrow& t-s>d(x, y),\\
    (x, s)\leq (y, t) &\Leftrightarrow& t-s\geq d(x, y),
  \end{eqnarray*}
  as well as a (metric) distance and Lorentzian distance defined by 
  \begin{eqnarray*}
    D((x, s),(y, t)) &= &\sqrt{ \vert t-s\vert^2 + d(x,y)^2},\\
    \tau ((x, s),(y, t)) &= &
                              \begin{cases}
                                \sqrt{ (t-s)^2 - d(x,y)^2 }& \hbox{if}\; (x, s)\le (y, t)\\  0 & \textrm{otherwise}.
                              \end{cases} 
  \end{eqnarray*}
}

In the context of (pre-)length spaces we can define the notion of timelike, causal and null curves as follows:

\begin{definition}
  \label{def:main:5}
  Let $(X,\ll,\leq,d,\tau)$ be a Lorentzian pre-length space. A non-constant locally Lipschitz continuous curve $\gamma:I\rightarrow X$ is \emph{future-directed timelike (resp. causal)} if $\gamma(t_1)\ll \gamma(t_2)$ (resp. $\gamma(t_1)\leq \gamma(t_2)$) for any $t_1<t_2$. A future-directed causal curve is \emph{null} if, in addition, $\gamma(t_1)\not\ll\gamma(t_2)$ for any $t_1<t_2$.
\end{definition}
With these concepts at hand, our aim now is to define the notion of timelike and causal geodesic. In this goal, first we need to define a notion of length in Lorentzian pre-length spaces. For this, we follow an analogous approach to the Riemannian case:
\begin{definition}
  \label{def:main:7}
  Let $(X,\ll,\leq,d,\tau)$ be a Lorentzian pre-length space and consider  a future-directed causal curve $\gamma$. We define the \emph{length} of $\gamma$, denoted by $L_{\tau}(\gamma)$, as follows:
  \[L_{\tau}(\gamma)=\mathrm{Inf}\left\{ \sum_i \tau(\gamma(t_i),\gamma(t_{i+1})):\; (t_i) \hbox{ is a partition of $I$} \right\}.\]
\end{definition}
Now we are in conditions to define the notion of geodesic:
\begin{definition}
  \label{def:main:6}
  Let $(X,\ll,\leq,d,\tau)$ be a Lorentzian pre-length space. A future-directed timelike (or causal) curve $\gamma:I\rightarrow X$ is a \emph{geodesic} if for each $t\in I$ there exists a neighbourhood $[a,b]$ of $t$ (with $a<t<b$ but allowing equality at the endpoints of $I$) such that $\gamma|_{[a,b]}$ is a \emph{distance realizer}, that is, $\tau(\gamma(a),\gamma(b))=L_{\tau}(\gamma|_{[a,b]})$. {A timelike (or causal) \textit{line} is an inextensible timelike curve that realizes distance among any pair of its points. An inextensible timelike (or causal) curve is said to be \textit{complete} if $L_\tau (\gamma )=\infty$.}
\end{definition}

\begin{definition}
  \label{def:main:15}
  A Lorentzian pre-length space $(X,\ll,\leq,d,\tau)$ is {\textit{causally geodesically connected} if for each pair $x,y\in X$ with $x\leq y$, there exists a geodesic $\gamma$ joining them.} Further, if for all $x\ll y$, there exists a timelike geodesic joining them, then the Lorentzian length space is  timelike geodesically connected. 
\end{definition}

Finally, let us end this section by giving the concept of Lorentzian length space. This notion requires to introduce first the following definition, which formalizes the idea of having the Lorentzian pre-length space good local properties:

\begin{definition}\cite[Defn. 2.15]{BeranEtal2023}
  \label{def:Escrituradelteorema:1}
  A Lorentzian pre-length space $(X,d,\ll,\leq,\tau)$ is called \emph{localisable} if for each $x\in X$ there exists a \emph{localising neighbourhood} $\Omega_x$ with the following properties:
  \begin{enumerate}[label=(\roman*)]
  \item the causal curves in $\Omega_x$ have uniformly bounded $d$-length, i.e., $\Omega_x$ is $d-$compatible;
  \item for each $y\in \Omega_x, I^{\pm}(y)\cap\Omega_x\neq \emptyset$;
  \item there is a so-called \emph{local time separation} $\omega_x:\Omega_x\times \Omega_x\rightarrow [0,\infty)$ such that $\Omega_x$ is a Lorentzian pre-length space upon restricting $d,\ll,\leq$;
  \item for all $p,q\in \Omega_x$ with $p<q$ there exists a causal curve $\gamma_{pq}$ from $p$ to $q$ entirely contained in $\Omega_x$ with maximal $\tau$-length among all causal curves from $p$ to $q$ contained in $\Omega_x$, as well as $L_{\tau}(\gamma_{pq})=\omega_x(p,q)$.
  \end{enumerate}
  If, in addition, for any $p,q\in \Omega_x$ with $p\ll q$ the curve $\gamma_{pq}$ is timelike and strictly longer than any causal curve from $p$ to $q$ in $\Omega_x$ containing a null segment, then $\Omega_x$ is called a \emph{regular localising neighbourhood}. If any point of $X$ has a regular localising neighbourhood then $X$ is said \emph{regularly localisable}.
\end{definition}
Previous notion also motivates the following definition (which will be used later):
\begin{definition}\cite[Defn. 2.19]{BeranEtal2023}
  \label{def:Escrituradelteorema:2}
  A localisable Lorentzian pre-length space $X$ is said to have the \emph{timelike geodesic prolongation} property if any maximising timelike segment $\gamma:[a,b]\rightarrow X$ can be extended to a timelike geodesic defined on an open domain, i.e., there is $\epsilon>0$ and a timelike geodesic $\tilde{\gamma}:(a-\epsilon,b+\epsilon)\rightarrow X$ such that $\tilde{\gamma}|_{[a,b]}=\gamma$.
\end{definition}
We are now in conditions to give the definition of Lorentzian length space:

\begin{definition}\cite[Defn. 2.24]{BeranEtal2023}
  \label{def:Escrituradelteorema:4}
  A Lorentzian pre-length space $(X,d,\ll,\leq,\tau)$ which is locally causally closed, causally path-connected and localisable is a \emph{Lorentzian length space} if, for any $x\leq y$ with $x\neq y$,
  \begin{equation}
    \label{eq:1}
    \tau(x,y)= \mathrm{sup}\left\{ L_{\tau}(\gamma):\gamma \hbox{ future-directed causal curve from $x$ to $y$} \right\}.
  \end{equation}

\end{definition}

\subsection{Curvature bounds}
In this section we will introduce the notion of curvature bounded from below and from above in the context of Lorentzian pre-length spaces. These notions are defined by using three model spaces: the ones corresponding to positive, negative and zero curvature. However, here we will be only interested in comparison with the zero curvature model, i.e. Minkowski spacetime.

There are two equivalent ways to formalize the notion of bounded curvature in this context, depending on the use of triangles or angles to make the comparison. Our arguments will require both approaches, so we will introduce them. For the purpose of this paper, we will make some minor adaptations to make the notions more suitable for Lorentzian metric spaces when necessary.

\begin{definition}
  \label{def:main:8}
  Let $(X,d,\ll,\leq,\tau)$ be a Lorentzian pre-length space. We will say that a triple $\triangle(p_1,p_2,p_3)$ is a \emph{timelike triangle} if $p_1\ll p_2\ll p_3$. We will call $p_1$ the \emph{past endpoint} and $p_3$ the \emph{future endpoint} of the triangle.

  We will name \emph{timelike geodesic triangle} to a timelike triangle $\triangle(p_1,p_2,p_3)$ together with timelike distance realizing curves $\gamma_{ij}$ connecting $p_i$ with $p_j$. For simplicity, we will denote again by $\triangle(p_1,p_2,p_3)$ the timelike geodesic triangles. 
\end{definition}

\begin{definition}
  \label{def:main:9}
  Let $(X,d,\ll,\leq,\tau)$ be a Lorentzian pre-length space. We will say that a timelike geodesic triangle $\overline{\triangle}(\overline{p}_1,\overline{p}_2,\overline{p}_3)$ in the $2$-dimensional Minkowski spacetime $\mathbb{L}^2$ is a comparison triangle for a timelike geodesic triangle  $\triangle(p_1,p_2,p_3)$ in $X$ if $\tau(p_i,p_j)=\overline{\tau}(\overline{p}_i,\overline{p}_j)$ for any $i<j$, being $\overline{\tau}$ the Lorentzian distance for $\mathbb{L}^2$.

  Given a point $x$ in the side $\gamma_{ij}$ of $\triangle (p_1,p_2,p_3)$, we say that a point $\overline{x}$ in the side $\overline{\gamma}_{ij}$ of $\overline{\triangle}(\overline{p}_1,\overline{p}_2,\overline{p}_3)$ is {\em the corresponding point of} $x$ if  $\tau(p_i,x)=\overline{\tau}(\overline{p}_i,\overline{x})$ and $\tau(x,p_j)=\overline{\tau}(\overline{x},\overline{p}_j)$. 
\end{definition}

Comparison triangles allow us to give a notion of bounded curvature for Lorentzian pre-length spaces. Actually, the notion of bounded curvature that we are going to use here is global and adapted to globally hyperbolic pre-length spaces, avoiding any extra technicality that we do not require. In particular, our notion is less general than the one presented in \cite{timelikecurv}. 

\begin{definition}
  \label{def:main:10}
  Let $X$ be a causally geodesically connected Lorentzian pre-length space and assume that $\tau$ is continuous. We will say that $X$ {\em has timelike curvature bounded below} by $0$ if the following property holds: for any timelike geodesic triangle $\triangle(p_1,p_2,p_3)$ and a comparison triangle $\overline{\triangle}(\overline{p}_1,\overline{p}_2,\overline{p}_3)$ in $\mathbb{L}^2$, for any pair of points $x,y$ on the sides of $\triangle(p_1,p_2,p_3)$, with $\overline{x},\overline{y}$ being the corresponding points in $\overline{\triangle}(\overline{p}_1,\overline{p}_2,\overline{p}_3)$, we have that
  \begin{equation}
    \label{eq:4}
    \tau(x,y)\leq \overline{\tau}(\overline{x},\overline{y}). 
  \end{equation}
\end{definition}

Next, we give the definition of bounded curvature by considering comparison angles, which are defined as follows.

\begin{definition}
  \label{def:main:11}
  Let $X$ a Lorentzian pre-length space, and consider a triangle $\triangle(p_1,p_2,p_3)$ and the co\-rres\-pon\-ding comparison triangle $\overline{\triangle}(\overline{p}_1,\overline{p}_2,\overline{p}_3)$ in $\mathbb{L}^2$. Then, we define the $0$-{\em comparison angle at} $p_1$ as:
  \begin{equation}
    \label{eq:5}
    \widetilde{\measuredangle}_{p_1}(p_2,p_3):=\overline{\measuredangle}_{\overline{p}_1}(\overline{p}_2,\overline{p}_3)\left( :=\langle\dot{\overline{\gamma}}_{12}(0),\dot{\overline{\gamma}}_{13}(0)\rangle \right),
  \end{equation}
  where $\langle\cdot,\cdot\rangle$ denotes the Minkowski metric. We also define the {\em signed angle} as $\widetilde{\measuredangle}^S_{p_1}(p_2,p_3):=-\widetilde{\measuredangle}_{p_1}(p_2,p_3)$\footnote{Observe that, according to the original definition \cite[Definition 2.1]{BeranSamann2023}, the signed angle depends on the vertex considered as the center of the angle. However, for the purpose of this paper, we will always consider angles either in the future or the past of the triangle, so the signed angle will always change the sign.}.
\end{definition}
Since the lengths of the sides of the comparison triangles coincide with those of the original triangles, previous definition of (hyperbolic) angle allows us to recover the law of cosines for angles in a Lorentzian pre-length space. Concretely, for a triangle $\triangle(p_1,p_2,p_3)$, if we denote $\alpha=\tilde{\measuredangle}_{p_1}(p_2,p_3)$, it follows that:
\begin{equation}
  \label{eq:6}
  \tau(p_1,p_3)^2= \tau(p_1,p_2)^2+\tau(p_2,p_3)^2- 2\tau(p_1,p_2)\tau(p_2,p_3)\cosh(\alpha)
\end{equation}
(here we have already taken into account the sign corresponding to the angle). As we can see from Figure \ref{fig:triangles}, in the smooth setting the curvature of the Lorentzian manifold imposes a concrete behaviour on the angle centered at a point, specially as we move closer to the point. So, in order to see how the bounded curvature affects angles, we need to introduce also the notion of angle between geodesics:

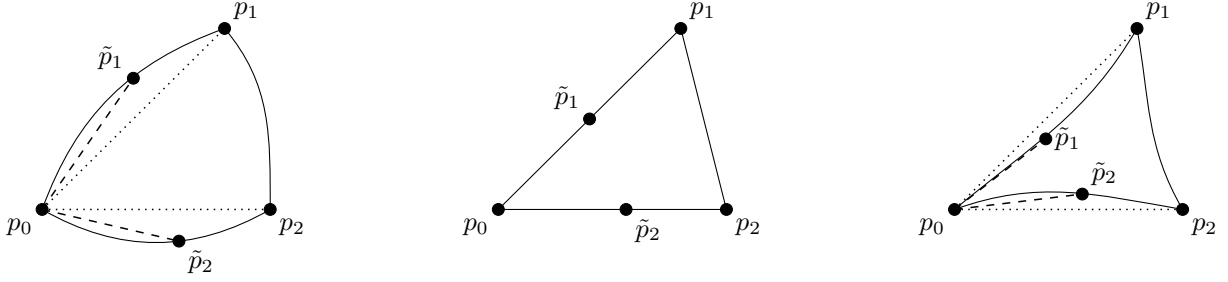
\begin{figure}
  \centering
  \begin{tikzpicture}[scale=1.2]
    \path [out=70, in=200, name path=A] (-5,0) edge (-3.,2);
    \path [out=-30, in=210, name path=B] (-5,0) edge (-2.5,0);
    \path [out=-50, in=90, name path=C] (-3,2) edge (-2.5,0);
    \filldraw[fill=black] (-5,0) circle (0.65mm)node[below left] {$p_{0}$};
    \filldraw[fill=black] (-3,2) circle (0.65mm)node[above right] {$p_{1}$};
    \filldraw[fill=black] (-2.5,0) circle (0.65mm)node[below right] {$p_{2}$};
    \path [dotted, line width = 0.2mm](-5,0) edge (-3,2);
    \path [dotted, line width = 0.2mm](-5,0) edge (-2.5,0);
    \path [dashed, line width = 0.2mm](-5,0) edge (-4,1.45);
    \path [dashed, line width = 0.2mm](-5,0) edge (-3.5,-0.35);
    \filldraw[fill=black] (-4,1.45) circle (0.65mm)node[above left] {$\tilde{p}_{1}$};
    \filldraw[fill=black] (-3.5,-0.35) circle (0.65mm)node[below right] {$\tilde{p}_{2}$};
    \path  (0,0) edge (2,2);
    \path  (0,0) edge (2.5,0);
    \path  (2,2) edge (2.5,0);
    \filldraw[fill=black] (0,0) circle (0.65mm)node[below left] {$p_{0}$};
    \filldraw[fill=black] (2,2) circle (0.65mm)node[above right] {$p_{1}$};
    \filldraw[fill=black] (2.5,0) circle (0.65mm)node[below right] {$p_{2}$};
    \filldraw[fill=black] (1,1) circle (0.65mm)node[above left] {$\tilde{p}_{1}$};
    \filldraw[fill=black] (1.4,0) circle (0.65mm)node[below right] {$\tilde{p}_{2}$};
    \path [out=40, in=240] (5,0) edge (7.,2);
    \path [out=20, in=170] (5,0) edge (7.5,0);
    \path [out=-80, in=120] (7,2) edge (7.5,0);
    \filldraw[fill=black] (5,0) circle (0.65mm)node[below left] {$p_{0}$};
    \filldraw[fill=black] (7,2) circle (0.65mm)node[above right] {$p_{1}$};
    \filldraw[fill=black] (7.5,0) circle (0.65mm)node[below right] {$p_{2}$};
    \filldraw[fill=black] (6,0.78) circle (0.65mm)node[right] {$\tilde{p}_{1}$};
    \filldraw[fill=black] (6.4,0.17) circle (0.65mm)node[above right] {$\tilde{p}_{2}$};
    \path [dotted, line width = 0.2mm](5,0) edge (7,2);
    \path [dotted, line width = 0.2mm](5,0) edge (7.5,0);
    \path [dashed, line width = 0.2mm](5,0) edge (6,0.78);
    \path [dashed, line width = 0.2mm](5,0) edge (6.4,0.17);
  \end{tikzpicture}
  \caption{\label{fig:triangles} The figure illustrates three triangles, each representing a different curvature: positive, zero, and negative. In these triangles we can see how the behaviour of the angles centered on $p_0$ change when we consider points $\tilde{p}_{1}$ and $\tilde{p}_2$ closer to $p_0$. }
\end{figure}

\begin{definition}
  \label{def:main:12}
  Let $p_0\in X$ be a point, and $\alpha,\beta:[0,\epsilon)\rightarrow X$ two future-directed timelike geodesics with $\alpha(0)=\beta(0)=p_0$. Assume that $I^+(\alpha)\subset I^+(\beta)$, and let $A_0:=\left\{ (s,t)\in (0,\epsilon)^2: \beta(t)\ll \alpha(s)  \right\}$. Then, we define the {\em upper angle between} $\alpha$ {\em and} $\beta$ at $p_0$ as
  \begin{equation}
    \label{eq:7}
    \measuredangle_{p_0}(\alpha,\beta)=\limsup_{(s,t)\in A_0;\,s,t\searrow 0}\widetilde{\measuredangle}_{p_0}(\alpha(s),\beta(t)).
  \end{equation}
  Accordingly, if previous limit exists, we will say that the angle
  between $\alpha$ and $\beta$ at $p_0$ exists. As in Definition \ref{def:main:11}, the signed angle is $\measuredangle_{p_0}^S(\alpha,\beta)=-\measuredangle_{p_0}(\alpha,\beta)$
\end{definition}

We are now in conditions to present the following consequence of \cite[Corollary 3.8 and Theorem 4.13]{BeranSamann2023}:

\begin{proposition}\label{thm:comparationangles}
  Let $(X,d,\ll,\leq,\tau)$ be a timelike geodesically connected Lorentzian pre-length space with timelike curvature bounded from below. Then, for any point $x\in X$ and any  future timelike geodesics $\alpha,\beta:[0,B]\rightarrow X$ with $\alpha(0)=\beta(0)=x$, it is satisfied
  \begin{equation}
    \label{eq:3}
    \measuredangle_x(\alpha,\beta)\geq \widetilde{\measuredangle}_x(\alpha(s),\beta(t)),
  \end{equation}
  for all $s,t\in [0,B]$ small enough such that $\alpha(s)$ and $\beta(t)$ are timelike related. An analogous version of the result can be obtained for past-directed timelike geodesics.
\end{proposition}

{Notions of curvature bounds from \textit{above} can be defined by reversing the inequalities \eqref{eq:4} and \eqref{eq:3}.} {Just as in the smooth scenario, non-negative timelike curvature bounds tends to hinder the formation of timelike lines in globally hyperbolic Lorentzian length spaces. Thus, lines can only exist in this context under very special circumstances. Namely, when the space splits. The following result can be considered a rough version of the classical Lorentzian Splitting Theorem \cite{Esch88,Gal89}.}

\begin{theorem}\label{thm:splitting}
  {\em \cite[Theorem 1.4]{BeranEtal2023}}
  Let $(X,d,\ll,\leq,\tau)$ be a connected regularly localisable, globally hyperbolic Lorentzian length space with proper metric $d$ and global non-negative timelike curvature satisfying the timelike geodesic prolongation and containing a complete timelike line $\gamma:\mathbb{R}\rightarrow X$. Then there is a $(\tau,\leq)$-preserving homeomorphism $f:\Sigma\times\mathbb{R}\rightarrow X$, where $\Sigma$ is a proper, strictly intrinsic metric of Alexandrov curvature $\geq 0$. 
\end{theorem}

\section{Structure of the c-boundary}\label{sec:cbdry}

The proof of the main result of this paper will require to show that, under some mild conditions including certain curvature bound, the future causal boundary of a Lorentzian pre-length space is a single point. In this section we will obtain some technical results oriented towards this purpose. When possible, we will try to prove them in full generality, i.e. in the context of Lorentzian metric spaces.

\begin{proposition}
  \label{prop:main:1}
  Let $(X,\sigma,\tau)$ be a chronologically dense Lorentzian metric space with a compatible causal relation $\leq$. Assume that $X$ is causally simple and its future causal boundary is formed by one point. If $I^-(y)\neq \emptyset$ for any $y\in X$, then any future causal chain $\varsigma=\left\{ z_k \right\}$, with $z_k\not\ll z_{k'}$ for any $k'>k$, is convergent to a point $x\in X$ so $I^-(\varsigma)=I^-(x)$.
\end{proposition}

\begin{proof}
  Assume by contradiction the existence of a future causal chain $\varsigma=\{z_k\}$ with $z_k\not\ll z_{k+1}$ which does not converge on $X$.

  Since $I^-(z_k)\neq \emptyset$ for all $k$, the  chronological density of $X$ ensures the existence of a future-directed timelike chain $\left\{ y_n^k \right\}_n$ for each $k$ such that $y_n^k\nearrow z_k$. Now, we construct inductively a chronological chain $\left\{ x_k \right\}_k$ with $I^-(\varsigma)=I^-(\left\{ x_k \right\}_k)$  as follows: since $y_n^k\nearrow z_k$, the causal character of $\varsigma=\{z_k\}$ joined to the push-up property guarantees the existence of $n_k$ such that  $y_{k}^1,\ldots,y^{k}_{k},x_{k-1}\ll y^k_{n_k}$; then, define $x_1:=y_1^1$, and given $x_{k-1}$ define $x_k:=y^k_{n_k}$. In particular,   given $r,l\in \mathbb{N}$ arbitrary, there exists $k_0$ big enough such that  $y_r^l\ll x_k$ for any $k\geq k_0$. So, we have both,  $I^-(\left\{ x_k \right\})\subset I^-(\varsigma)$ (note that $x_k=y^k_{n_k}\ll z_k$) and  $I^-(\varsigma)\subset I^-(\left\{ x_k \right\})$ (note that any point $z\in I^-(\varsigma)$ belongs to the past of $y_l^r$ for appropriate $l$, $r$). Therefore, $I^-(\varsigma)=I^-(\left\{ x_k \right\})$.

  Moreover, as $z_k\not\ll z_{k+1}$, any point $z\in I^+(z_k)$ for some $k$ cannot be contained in $I^-(\varsigma)$, hence $I^-(\varsigma)\neq X$. Consequently, $I^-(\varsigma)$ should be a PIP, that is, $I^-(\varsigma)=I^-(x)$ for some $x\in X$. Moreover, since $\left\{ x_k \right\}$ is a chronological chain with $I^-(\left\{ x_k \right\})=I^-(x)$, it follows that $\{x_k\}\rightarrow x$ with $\sigma$. To finish the proof we need to show that $\varsigma\rightarrow x$, violating the non convergence of $\varsigma$.

  To this aim, let $U$ be a neighborhood of $x$, and let us prove that $z_k\in U$ for any $k$ big enough. Since $X$ is a causally continuous, thus strongly causal, spacetime, the topology $\sigma$ coincides with the Alexandrov one. So, there exist $p^{\pm}\in X$ such that $x\in I(p^-,p^+)\subset U$. On the other hand, since $I^-(\varsigma)=I^-(x)$, necessarily $\varsigma\subset \overline{I^-(\varsigma)}=\overline{I^-(x)}=J^-(x)$, where the last equality follows from causal simplicity. Therefore, $z_k\leq x\ll p^+$ for any $k$, and the push-up property gives $z_k\ll p^+$ for any $k$. Finally, since $x_k\rightarrow x\subset I(p^-,p^+)$, necessarily  $x_{k_0}\in I(p^-,p^+)$ for some $k_0$. So, taking into account that, by construction, $x_{k_0}=y^{k_0}_{n_{k_0}}\ll z_{k_0}$, necessarily $x_k\ll z_{k_0}\leq z_k$ for any $k\geq k_0$, and again the push-up property gives $p^-\ll z_k$ for any $k\geq k_0$. In conclusion, $z_k\in I(p^-,p^+)\subset U$ for $k\geq k_0$, and using  that $U$ is an arbitrary neighborhood of $x$, we conclude that $z_k\rightarrow x$, as required.
\end{proof}

As a straightforward consequence we {have the following:}

\begin{corollary}\label{coro:nonull}
  Any causally simple Lorentzian pre-length space whose future causal boundary is formed by one point has no inextensible null lines.
\end{corollary}

{We close this section with a structure result  that will be fundamental in the proof of BSC in this setting.}

\begin{proposition}
  \label{thm:main:1}
  Let $(X,\sigma,\tau)$ be a Lorentzian metric space, where $X=\Sigma\times\mathbb{R}$ and $\sigma$ is the product topology. Assume also that $\tau((x,t),(x,s))>0$ for any $t,s\in \mathbb{R}$ with $t<s$.
  If $\varsigma=\left\{ (x_n,t_n) \right\}$ is a chronological chain with $t_n\nearrow \infty$ and $x_n\rightarrow x_0\in X$, then $I^-(\varsigma_{x_0})\subset I^-(\varsigma)$, where $\varsigma_{x_0}:=\left\{ (x_0,k) \right\}_k$.

  In particular, if $\Sigma$ is compact then any TIP contains the past of $\varsigma_{x_0}$ for some $x_0\in \Sigma$.
\end{proposition}
\begin{proof}
  The result is quite straightforward from the hypotheses once we recall that the chronological relation is open. In fact, given $(x_0,k)\in \Sigma\times\mathbb{R}$, there exists  an open neighborhood $V\subset\Sigma$ of $x_0$ and some $\epsilon>0$ such that $V\times (-\epsilon+k+1,\infty)\subset I^+((k,x_0))$. From the hypothesis about $\varsigma$, there exist $n_0$ big enough such that $t_{n_0}>-\epsilon+k+1$ and $x_{n_0}\in V$. Hence, $(t_n,x_n)\in I^+(k,x_0)$ (i.e. $(k,x_0)\ll (t_n,x_n)$) for any $n>n_0$. Since previous relation holds for any $k$, we deduce that $I^-(\varsigma_{x_0})\subset I^-(\varsigma)$.

  For the last statement, note that, if $\Sigma$ is compact, the sequence $\{x_n\}$ associated to 
  any future chronological chain $\varsigma=\left\{ (t_n,x_n) \right\}$ is convergent, up to a subsequence. So, the conclusion follows by noting that any TIP is the past of some chain like $\varsigma$.
\end{proof}

\section{A low regularity version of BSC}\label{sec:main}

{Notice that Proposition \ref{thm:main:1}} ensures that, if $\Sigma$ is compact, any TIP includes the past of any ``vertical line'' of $X$. Therefore, if we prove that the past of any of these lines is $X$, then the future causal boundary will consist of a single point (and Proposition \ref{prop:main:1} will apply). However, this is not true in general (see, for instance, {\cite[Section 4.2]{AlFlores}} for cases where each vertical line defines a different TIP), and consequently, an additional hypothesis is required. {A timelike curvature bound from below provides such a condition in the context of} Lorentzian \emph{pre-length} spaces. In fact, the main result of this section is the following:

\begin{theorem}
  \label{thm:main:3}
  Let $(X,\ll,\leq,d,\tau)$ be a strongly causal, causally closed, timelike geodesically connected Lorentzian length space with $\tau$ continuous and having timelike curvature bounded from below by zero. {Assume that $X=\Sigma\times \mathbb{R}$ splits as a topological product, and that the vertical lines are timelike complete} with $\tau((x,t),(x,s))>0$ for any $t,s\in \mathbb{R}$ and $t<s$. Then, the past of any vertical line $\gamma(s)=(b,s)$ with $b\in \Sigma$ and $s\in [0,\infty)$, is the entire set $X$. In particular,  if $\Sigma$ is compact then the future causal boundary of $X$ is a single point.
\end{theorem}
\begin{proof} Observe that the past of any future-directed vertical line contains the line itself. So, in order to prove that the past of any such line is the entire $X$, it suffices to show that they all share the same past. To this aim, consider two arbitrary vertical lines, say $\delta,\gamma:[0,\infty)\rightarrow\Sigma\times\mathbb{R}$, $\delta(s)=(a,s)$, $\gamma(s)=(b,s)$, $a\neq b$. First, we will assume that $a,b$ satisfy $(a,0)\in I^-(\gamma)$ and $(b,0)\in I^-(\delta)$. This happens, for instance, if $a,b$ are close enough between them, since $(a,0)\in I^-(\delta)$, $(b,0)\in I^-(\gamma)$, and $I^-(\delta)$, $I^-(\gamma)$ are both open sets.

  Assume by contradiction that $I^-(\gamma)$ does not contain entirely the curve $\delta$. Next, we are going to define two sequences $\left\{A_i \right\}\subset \mathrm{Im}(\delta)$ and $\left\{ B_i \right\}\in \mathrm{Im}(\gamma)$ contained in the image of $\delta$ and $\gamma$, respectively. Since $\lim_{s\rightarrow \infty}\tau((x,t),(x,s))=\infty$, we can define the sequence $\left\{ B_i \right\}$ so that $\tau(B_i,B_{i+1})>1$ for any $i$; so, by the reverse triangle inequality,  $\tau(B_i,B_j)>j-i$ for any $j>i$. We will also take $B_1$ so that $A_1:=(a,0)\ll B_1$. For each $i\in \mathbb{N}$, choose $A_i$ so that $0<\tau(A_i,B_i)<1/i$ as follows: since $(a,0)=A_1\ll B_1\ll B_2$, necessarily $\tau(A_1,B_2)>0$. Since $\tau$ is continuous, there exists $s'_{2}\in \mathbb{R}$ such that $\tau((a,s'_2),B_2)=0$ (otherwise, $\tau((a,s),B_2)>0$ for all $s$ and $\delta\subset I^-(\gamma)$). So, we can take $A_2=(a,s_{2})$ with $s_2<s'_2$ close enough so that $0<\tau(A_2,B_2)<\frac{1}{2}$. By repeating this process iteratively, we construct the required sequence $\{A_i=(a,s_i)\}$.

  By construction, the sequence $\{s_i\}$ of $\{A_i=(a,s_i)\}$ should converge. In fact, otherwise, since $\{s_i\}$ is increasing, we deduce that $\delta$ is contained in the past of $\gamma$, a contradiction. Therefore, since $\tau$ is continuous, the values $\tau(A_i,A_j)$ for any $i,j\in \mathbb{N}$ are uniformly bounded (see the picture on the left of Figure \ref{fig:1}).

  \begin{figure}
    \centering
    \begin{tikzpicture}
      [
      punto/.style = {circle, fill=black, scale=0.5},
      puntop/.style = {circle, fill=black, scale=0.3},
      every node/.style ={scale=1.3}
      ]
      \draw [->] (0,0) -- node [left, pos =0.95] {$\delta$} (0,7.5);
      \draw [->] (4,0) -- node [right, pos = 0.95] {$\gamma$} (4,7.5);
      \coordinate [punto, label= left:{$A_1$}] (A1) at (0,0);
      \coordinate [punto, label= left:{$A_i$}] (Ai) at (0,1.6);
      \coordinate [punto, label= left:{$A_2$}] (A2) at (0,0.9);
      \coordinate [puntop] (none) at (0,1.9);
      \coordinate [puntop] (none) at (0,2.1);
      \coordinate [punto] (limitp) at (0,2.3);
      \coordinate [punto] (none) at (4,2);
      \coordinate [punto, label= right:{$B_2$}] (B2) at (4,3);
      \coordinate [punto] (B1) at (4,4);
      \coordinate [punto] (none) at (4,5);
      \coordinate [punto, label = right:{$B_i$}] (Bi) at (4,6);
      \draw [thick,dashed, in=240, out=25,color=red] (A1) to  (B2);
      \draw [thick,dashed,color=red] (A2) to [bend right=5] (B2);
      \draw [thick,dashed, in=240, out=50,color=blue] (A1) to  (Bi);
      \draw [thick,dashed,color=blue] (Ai) to [bend right = 10] (Bi);
      \draw [thick, dotted, in=267, out=45] (limitp) to (3.8,7.5);
      \draw [thick, dashed, color=red] (A1) to (A2);
      \draw [thick, dashed, color=blue] (A1) to (Ai);
      \centerarc[thick](0,0)(50:90:0.7cm);
      \coordinate [label= right:{$\alpha_{i}$}] (alpha) at ({0.7*cos(55)},{0.7*sin(55)});
      \coordinate [punto, label= left:{$\bar{A}_1=(0,0)$}] (A1b) at (10,0);
      \coordinate [punto, label= right:{$\bar{B}_{2}$}] (B2b) at (14,3);
      \coordinate [punto, label= left:{$\bar{A}_{2}$}] (A2b) at (10,1.3);
      \coordinate [punto, label= left:{$\bar{A}_{i}$}] (Aib) at (10,3.3);
      \coordinate [punto, label= right:{$\bar{B}_{i}$}] (Bib) at (14,6);
      \draw [thick, dashed, color=red] (A1b) to (B2b) to (A2b) to (A1b);
      \draw [thick, dashed,color=blue] (A1b) to node [pos=0.7, right,color=black] {\small$\tau(A_1,B_i)$} (Bib) to (Aib) to node [pos=0.3,left,color=black] {\small$\tau(A_1,A_i)$} (A1b);
      \centerarc[thick](10,0)(55:90:0.7cm);
      \coordinate [label = left:{$\overline{\alpha}_i$}] (alpha2) at ({10+0.6*cos(90)},{0.6*sin(90)});    
    \end{tikzpicture}

    \caption{\label{fig:1} The picture on the left shows the curves $\delta$ and $\gamma$ with the corresponding points $\left\{ A_i \right\}$ and $\left\{ B_i \right\}$. By construction, $\{A_i\}$ converges to a point in $\delta$ and $\left\{ B_i \right\}$ satisfies $\tau(B_i,B_{i+1})>1$ and $A_i\ll B_i$ for all $i$. In addition,  $\left\{\tau(A_i,B_i) \right\}$ converges to $0$.\newline
      The picture on the right illustrates two comparison triangles in Minkowski spacetime. The triangle $\triangle{(\bar{A}_1\bar{B}_i\bar{A}_i)}$ satisfies that the length of the segment determined by $\bar{A}_1$ and $\bar{B}_i$ is $\tau(A_{1},B_i)$, the length of the segment determined by $\bar{A}_1$ and $\bar{A}_i$ is $\tau(A_1,A_i)$, and the angle determined by both segments is $\bar{\alpha}_i$. \newline
      If we assume that the timelike curvature of $X$ is bounded from below by $0$, necessarily $\bar{\alpha}_i\leq \alpha_i$ for all $i$, where $\alpha_i= \widetilde{\measuredangle}_{A_1}(A_i,B_i)$. 
    }
  \end{figure}
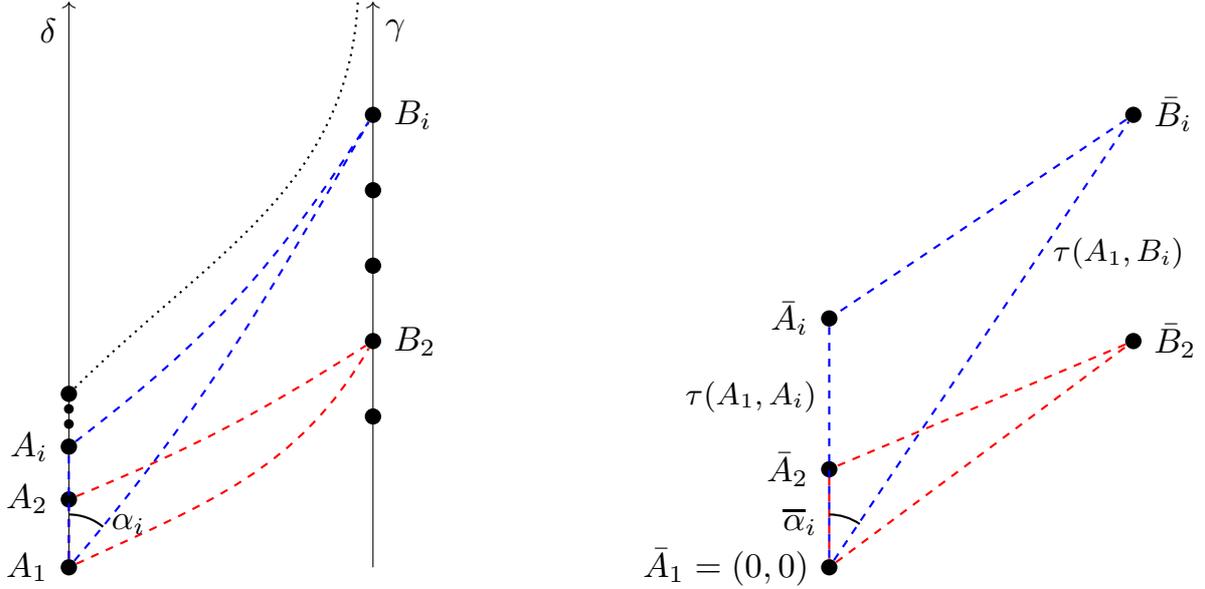

  \smallskip

{Denote by $\delta_i$ a (distance-realizing) timelike geodesic connecting $A_1$ to $A_i$, denote by
    $\lambda_i$ a timelike geodesic from $A_1$ to $B_i$,} and define $\alpha_i:=\measuredangle_{A_1}(\delta_i,\lambda_i)$. Observe that $\delta_i$ is not necessarily a segment of $\delta$, since $\delta$ is not necessarily a geodesic. By the Limit Curve Lemma \cite[Theorem 2.23]{BeranEtal2023}, there exists a future directed causal limit curve $\lambda$ emanating from $A_1$; indeed, $\lambda$ is a geodesic (since all $\lambda_i$ are).  Moreover,  since $\tau (A_1,B_i)\ge \tau (A_1,B_1)+\tau (B_1,B_i)>i-1$ we have (see \cite[Prop. 3.17]{KunzingerSamann2018}):
  \[
    L_\tau (\lambda )\ge\limsup L_\tau ( \lambda_i ) =\limsup \tau (A_1,B_i) =\infty ,
  \]
  which implies that $\lambda$ is timelike. Moreover, by continuity of angles (see \cite[Prop. 4.15]{BeranSamann2023}), $\alpha_i\to \measuredangle_{A_1}(\delta ,\lambda)$, and consequently the angle $\measuredangle_{A_1}(\delta ,\lambda)$ (between two future directed timelike geodesics) is finite. 

  Now, consider a comparison triangle $\triangle(\bar{A}_{1}\bar{A}_{i}\bar{B}_{i})$ for $\triangle (A_1A_iB_i)$, and denote $\overline{\alpha}_{i}=\bar{\measuredangle}_{\overline{A}_1}(\overline{A}_i\overline{B}_i)$. Taking into account that $X$ has timelike curvature bounded from below, Proposition \ref{thm:comparationangles} ensures that $\alpha_i\geq \bar{\alpha}_i$, and recalling that $\{\alpha_i\}$ is convergent, we deduce that $\{\bar{\alpha}_i\}$ are bounded from above.

  Finally, if we apply the Law of Cosines \eqref{eq:6} to $\triangle(\bar{A}_{1}\bar{A}_i\bar{B}_i)$ then
  \[
    \bar{\tau}(\bar{A}_{i},\bar{B}_{i})^2=\bar{\tau}(\bar{A}_{1},\bar{A}_{i})^2+\bar{\tau}(\bar{A}_{1},\bar{B}_{i})^2-2\bar{\tau}(\bar{A}_{1},\bar{A}_{i})\bar{\tau}(\bar{A}_{1},\bar{B}_i)\cosh \bar{\alpha}_i.
  \]
  Recalling now that $\triangle(\bar{A}_{1}\bar{A}_{i}\bar{B}_{i})$ is a comparison triangle for $\triangle (A_1A_iB_i)$, and dividing by $\tau(A_1,B_i)$, we have
  \begin{equation}\label{kju}
    \frac{\tau({A_i},{B_i})^2}{\tau (A_1,B_i)}=\frac{\tau({A_1},{A_i})^2}{\tau (A_1,B_i)}+{\tau}({A_1},{B_i})-2{\tau}({A_1},{A_i})\cosh {\bar{\alpha}_i}.
  \end{equation}
  Therefore, if we take the limit in (\ref{kju}) as $i$ goes to infinity, and recall the inequality $\tau(A_1,B_i)>i-1$, and the fact that $\tau(A_i,B_i), \tau(A_1,A_i)$ and $\bar{\alpha}_i$ are bounded from above, we obtain that the left hand size tends to $0$, while the right hand side diverges, a contradiction.

  \smallskip

  In summary, we have showed that, for $a$, $b$ close enough, the curves $\delta$ and $\gamma$ have the same past. For $a,b$ arbitrary, a standard compact argument using a (continuous) curve $r:[0,1]\rightarrow \Sigma$ joining both points can be applied. In fact, let $\delta_t:[0,\infty)\rightarrow X$ be the vertical line $\delta_t(s)=(r(t),s)$. Since $(r(t),0)\in I^-(\delta_t)$ and the chronological relation is open, we can take open sets $r(t)\in U_t\subset I^-(\delta_t)$. The family $\left\{ U_t \right\}_{t\in [0,1]}$ covers the curve $r(t)$, which is compact, so it can be covered by a finite sub-family, say $\left\{ U_{t_i} \right\}_{i\in \left\{ 1,\dots,n \right\}}$. Finally, consider a finite sequence $\{\bar{t}_i\}$ such that $r(\bar{t}_i)\in U_{t_i}\cap U_{t_{i+1}}$ for $i=1,\dots,n-1$, and define $\bar{t}_0=0$ and $\bar{t}_n=1$. Previous argument shows that $I^-(\delta_{\bar{t}_i})=I^-(\delta_{\bar{t}_{i+1}})$ for all $i=0,\dots,n-1$. In particular, $I^-(\delta)=I^-(\gamma)$, as required. 
  
  For the last assertion, just apply Proposition \ref{thm:main:1}.
\end{proof}

Now we are ready to establish the main result of this work, namely the low regularity version of BSC elucidated in Theorem \ref{thm:main:2}. The key ingredients being Theorem \ref{thm:splitting} and Theorem \ref{thm:main:3}.

 \medskip

\noindent {\em Proof of Theorem \ref{thm:main:2}}.
  According to Theorem \ref{thm:splitting},	
  we only need to prove the existence of a timelike line. For this purpose, let  $\left\{ x^{\pm}_n \right\}$ be two sequences where $x_n^{\pm}=(p,\pm n)$, for some $p\in \Sigma$ and all $n$. Since $X$ is globally hyperbolic, there exists a sequence of timelike maximizing geodesics $\left\{ \gamma_n \right\}$, with $\gamma_n:[a_n,b_n]\rightarrow X$, such that $\gamma_n(a_n)=x_n^-$ and $\gamma_n(b_n)=x_n^{+}$. By construction, it is clear that $L_d(\gamma_n)\rightarrow \infty$, and we can take $c_n\in [a_n,b_n]$ such that $\gamma_n(c_n)\in \Sigma\times \left\{ 0 \right\}$. Taking into account that $\Sigma$ is compact, we can apply appropriately the limit curve theorem \cite[Theorem 2.23]{BeranEtal2023} to find a subsequence $\left\{ \gamma_{n_k} \right\}$ that $C^0$-converges to a complete causal line $\gamma:\mathbb{R}\rightarrow X$. Then, we fall under the hypotheses of Theorem \ref{thm:main:3}, which guarantees that the future causal boundary is a single point. So, by applying Corollary \ref{coro:nonull}, we deduce that the causal line $\gamma$ is necessarily timelike, which concludes to proof. \qed


\begin{thebibliography}{99}

\bibitem{AlFlores}
  V. Alaña, and J.L. Flores. \textit{The causal boundary of product spacetimes.} Gen. Rel. Grav. \textbf{39} pp. 1697–1718, 2007
  
\bibitem{AB}
  S. Alexander, and R. Bishop. \textit{Lorentz and semi-Riemannian spaces with Alexandrov curvature bounds.} Commun. Anal. Geom. \textbf{16}(2) pp. 251-282, 2008

\bibitem{ABS}
  L. Ake Hau, S. Burgos, and D.A. Solis. \textit{Causal completions as Lorentzian pre-length spaces.} Gen. Rel.  Grav. \textbf{54}  108, 2022

\bibitem{ACS}
  L. Ake Hau, A. Cabrera Pacheco, and D. A. Solis. \textit{On the causal hierarchy of Lorentzian length spaces.} Class. Quantum Grav. \textbf{37}(21), 2020

\bibitem{AKP}
  S. B. Alexander, V. Kapovitch, and A. Petrunin. \textit{Alexandrov Geometry: Foundations.} (Preprint)  https://arxiv.org/abs/1903.08539, 2023.

\bibitem{BMS2} W. Barrera, L. Montes de Oca, and D.A. Solis. \textit{On the space of compact diamonds of Lorentzian length spaces.} Class. Quantum Grav. \textbf{41} 065012, 2024

\bibitem{BarreraEtal2022}
  W. Barrera, L. Montes de Oca, and D. A. Solis. \textit{Comparison theorems for Lorentzian length spaces with lower timelike curvature bounds.} Gen. Rel. Grav. \textbf{54}(9) 107, 2022

\bibitem{Bartnik}
  R. Bartnik. \textit{Remarks on cosmological spacetimes and constant mean curvature surfaces.} Comm. Math. Phys. \textbf{117} pp. 615-624, 1988
  
\bibitem{timelikecurv}
  T. Beran, M. Kunzinger, and F. Rott. \textit{On curvature bounds in Lorentzian length spaces.} J.  London Math. Soc. \textbf{110}(2), 2024

\bibitem{Bonnet}
  T.  Beran, L. Napper, and F. Rott. \textit{Alexandrov's Patchwork and the Bonnet-Myers Theorem for Lorentzian length spaces.} (Preprint) arXiv:2302.11615, 2023

\bibitem{BeranEtal2023}
  T. Beran, A. Ohanyan, F. Rott, and D. A. Solis, \textit{The splitting theorem for globally hyperbolic Lorentzian length spaces with non-negative timelike curvature.} Lett. Math. Phys. \textbf{113}(2) 48, 2023.

\bibitem{BeranSamann2023}
  T. Beran, and C. S\"{a}mann. \textit{Hyperbolic angles in Lorentzian length spaces and timelike curvature bounds}. J.  London Math. Soc. \textbf{107}(5) pp. 1823-1880, 2023

\bibitem{Bridson}
  M. R. Bridson and A. Haefliger. \textit{Metric spaces of non-positive curvature}. Fundamental Principles of Mathematical Sciences, vol. 319, Springer, Berlin Heidelberg, 1999.

\bibitem{Burago}
  D. Burago, Y. Burago, and S. Ivanov. \textit{A course in metric geometry.} Graduate Studies in Mathematics, vol. 33, American Mathematical Society, Providence, RI, 2001

\bibitem{BGM}
  Y. Burago, M. Gromov, and G. Perelman. \textit{Alexandrov spaces with curvature bounded below.} Russian Math. Surveys \textbf{47}(2)  pp. 1-58, 1992

\bibitem{BFH}
  S. Burgos, J.L. Flores, and J. Herrera. \textit{The c-completion of Lorentzian metric spaces.} Class. Quantum Grav. \textbf{40} 20501, 2024

\bibitem{BH}
  A. Burtscher and L. Garc\'{\i}a-Heveling. \textit{Time functions on Lorentzian length spaces.} 2024.  https://doi.org/10.1007/s00023-024-01461 

\bibitem{M87}
  Collaboration, T.E.H.T. \textit{First M87 event horizon telescope results. I. The shadow of the supermassive black hole.} Astrophys. J. Lett. 875, pp 1-17, 2019

\bibitem{LIGO}
  Collaboration, L.S., Collaboration, V. \textit{Observation of gravitational waves from a binary black hole merger.} Phys. Rev. Lett. \textbf{116} 061102, 2016

\bibitem{CG}
  P.T. Chru\'sciel and J.D.E. Grant. \textit{On Lorentzian causality with continuous metrics} Class. Quantum Grav. \textbf{29} 145001, 2012


\bibitem{CSHF22}
  I.P. Costa e Silva, J.L. Flores and J. Herrera. \textit{Omniscient foliations and the geometry of cosmological spacetimes}, Gen. Rel. Grav. \textbf{54} 147, 2022.

  
\bibitem{Esch88}
  J.H. Eschenburg. \textit{The splitting theorem for space-times with strong energy condition.} J. Diff. Geom. \textbf{27} pp. 477-491, 1988

\bibitem{Gal89}
  G.J. Galloway. \textit{The Lorentzian splitting theorem without the completeness assumption.} J. Diff. Geom. \textbf{29} pp. 373-387, 1989

\bibitem{Gal84}
  G.J. Galloway. \textit{Splitting theorems for spatially closed spacetimes.} Comm. Math. Phys. \textbf{96} pp. 423-429, 1984

\bibitem{Gal18}
  G.J. Galloway and E. Ling. \textit{Existence of CMC Cauchy surfaces from a spacetime condition.} Gen. Rel. Grav. \textbf{50} 108, 2018

\bibitem{Geroch}
  R. Geroch. \textit{Domain on dependence.} J. Math. Phys. \textbf{11}, pp. 437-449, 1970

\bibitem{GKP}
  R. Geroch, E.H. Kronheimer, and R. Penrose. \textit{Ideal points in space-time.} Proc. Roy. Soc. Lond. A \textbf{327}(1571) pp 545-567, 1972

\bibitem{Graf}
  M. Graf. \textit{Singularity theorems for $C^1$-Lorentzian metrics.} Comm. Math. Phys.  \textbf{378} pp. 1417-1450, 2020

\bibitem{Hedicke}
  J. Hedicke. \textit{Lorentzian distance functions in contact geometry.} J. Topol. Anal. pp. 1-21, 2022

\bibitem{KOSS}
  M. Kunzinger, A. Ohanyan, B. Schinnerl and R. Steinbauer. \textit{The Hawking-Penrose singularity theorem for $C^1$-Lorentzian metrics.} Commun. Math. Phys. \textbf{391} 1143-1179, 2022

\bibitem{KunzingerSamann2018}
  M. Kunzinger and C. S\"{a}mann. \textit{Lorentzian length spaces} Ann. Global Anal. Geom. \textbf{54} pp. 399-447, 2018.

\bibitem{MS}
  E. Minguzzi and S. Suhr. \textit{Lorentzian metric spaces and their Gromov-Hausdorff convergence.} \textbf{114} 73, 2024

\bibitem{Mingcone}
  E. Minguzzi. \textit{Causality theory for closed cone structures with applications.} Rev. Math. Phys. \textbf{31} 1930001, 2019

\bibitem{Olaf}
  O. M\"uller. \textit{Functors in Lorentzian geometry: three variations on a theme}. Gen. Rel. Grav. \textbf{55} 39, 2023

\bibitem{Olafdim}
  O. M\"uller. \textit{Dimensions of ordered spaces and Lorentzian length spaces} (Preprint) arXiv:2303.11237, 2023

\bibitem{PS}
  J. Podolsk\'{y} and R. Steinbauer. \textit{Penrose junction conditions with $\Lambda$: geometric insights into low-regularity metrics for impulsive gravitational waves.} Gen. Rel. Grav. \textbf{54} 96, 2022

\bibitem{causalgluing}
  F. Rott. \textit{Gluing of Lorentzian length spaces and the causal ladder.} Class. Quantum Grav. \textbf{40}(17), 2023.

\bibitem{Surya}
  S. Surya. \textit{The causal set approach to quantum gravity.} Living Rev. Relativ. \textbf{22} 5, 2019



\end{thebibliography}
\end{document}